\documentclass[11pt,reqno,a4paper,twoside]{article}

\usepackage{amsmath,amsthm,amstext,amscd,amssymb,euscript,mathrsfs, url}
\usepackage{calrsfs}
\usepackage{epsfig}
\usepackage{tikz}
\usepackage{enumitem}

\newcommand{\R}{\mathbb R}
\newcommand{\N}{\mathbb N}

\newcommand{\E}{\mathbb E}

\renewcommand{\phi}{\varphi}

\newcommand{\pee}{\ensuremath{\mathbb{P}}}

\newcommand{\V}{\mathbb V}
\def\1{{\mathchoice {\rm 1\mskip-4mu l} {\rm 1\mskip-4mu l}
{\rm 1\mskip-4.5mu l} {\rm 1\mskip-5mu l}}}

\newtheorem{theorem}{{\small T}{\scriptsize HEOREM}}[section]
\newtheorem{corollary}{{\bf{\small C}{\scriptsize OROLLARY}}}[section]
\newtheorem{proposition}{{\bf{\small P}{\scriptsize ROPOSITION}}}[section]
\newtheorem{lemma}{{\bf{\small L}{\scriptsize EMMA}}}[section]
\newtheorem{remark}{{\bf{\small R}{\scriptsize EMARK}}}[section]
\newtheorem{definition}{{\bf{\small D}{\scriptsize EFINITION}}}[section]

\renewenvironment{proof}[1]
{\noindent{{\bf{\small{ P}{\scriptsize ROOF}}}.}\hspace{0.1cm} #1} {$\;\qed$\newline}

\newcommand{\beq}{\begin{eqnarray}}
\newcommand{\eeq}{\end{eqnarray}}

\newcommand{\ba}{\begin{align*}}
\newcommand{\ea}{\end{align*}}

\newcommand{\be}{\begin{equation}}
\newcommand{\ee}{\end{equation}}

\newcommand{\bl}{\begin{lemma}}
\newcommand{\el}{\end{lemma}}

\newcommand{\br}{\begin{remark}}
\newcommand{\er}{\end{remark}}

\newcommand{\bt}{\begin{theorem}}
\newcommand{\et}{\end{theorem}}

\newcommand{\bd}{\begin{definition}}
\newcommand{\ed}{\end{definition}}

\newcommand{\bp}{\begin{proposition}}
\newcommand{\ep}{\end{proposition}}

\newcommand{\bc}{\begin{corollary}}
\newcommand{\ec}{\end{corollary}}

\newcommand{\bpr}{\begin{proof}}
\newcommand{\epr}{\end{proof}}

\newcommand{\bi}{\begin{itemize}}
\newcommand{\ei}{\end{itemize}}

\newcommand{\ben}{\begin{enumerate}}
\newcommand{\een}{\end{enumerate}}

\newcommand{\caA}{{\mathcal A}}

\newcommand{\caE}{{\mathrsfs E}}

\newcommand{\caJ}{{\mathcal J}}
\newcommand{\caK}{{\mathcal K}}

\newcommand{\caS}{{\mathcal S}}

\newcommand{\aaaa}{\mathbb{A}}
\newcommand{\tee}{\mathbb{T}}
\begin{document}
\title{Generalized immediate exchange models and their symmetries}
\author{
Frank Redig
and Federico Sau
\\
\small{Delft Institute of Applied Mathematics}\\
\small{Delft University of Technology}\\
{\small Mekelweg 4, 2628 CD Delft}
\\
\small{The Netherlands}
}
\maketitle

\pagenumbering{arabic}

\begin{abstract}
We reconsider the immediate exchange model and define
a more general class of models where mass is split, exchanged and
merged. We relate the splitting process to the symmetric inclusion process via thermalization and
from that obtain symmetries and self-duality of the generalized $IEM$ model.
We show that analogous properties hold for models were
the splitting is related to the symmetric exclusion process
or to independent random walkers.
\end{abstract}
\section{Introduction}

The immediate exchange model is a model of wealth distribution, introduced in \cite{Heipat},  further
studied in \cite{pipo}, generalized and studied from the viewpoint of processes with duality in \cite{grs}.

In words, it is a model in which two agents at random exponential event times each split their wealth (a non-negative real quantity) into two parts, uniformly,  then exchange the ``top parts''
and add the two parts again to obtain their updated wealth.
The model conserves the total wealth and is reminiscent of models of statistical mechanics such as the $KMP$ model
and its generalizations \cite{kmp}, \cite{cggr}. Moreover, it has reversible product measures of type $\Gamma(2)$.\par
We showed in \cite{grs} that the splitting can be done according to a $B(s,t)$ distribution, and then
the model has reversible product measures of type $\Gamma(s+t)$. This was established using a duality with
a discrete model of the same type, where discrete mass is redistributed in an analogous way, and where
there the splitting part is using a Beta Binomial distribution. \newline
We proved that this discrete model is
self-dual, has reversible product measures which are discrete $\Gamma(s+t)$ distributions.
As a consequence, by considering a many-particle limit of this discrete model, one recovers the original continuous model, as well as the duality between these two models. \par
In this paper, we first give a new perspective on the discrete immediate exchange model, by viewing the \emph{splitting
part} of the dynamics as a thermalization of the symmetric inclusion process ($SIP$).
This immediately leads to symmetries of the splitting part. We then show that these symmetries are
permutation invariant, and therefore also commute with the \emph{exchange part} of the dynamics.
Remarkably, the symmetries also survive the \emph{addition part}, because the symmetries of the $SIP$
have a natural additive structure in the parameter labeling the representation.\newline
This recovers in a much more elegant way the full $SU(1,1)$ symmetry of the $IEM(1,1)$ model, where
the parameter of the discrete representation is $\kappa=2=1+1$,  arising as the addition of the parameters of the representations of the two underlying $SIP(1)$ processes, where $\kappa=1$.
This can then be immediately generalized to the $IEM(s_1, t_1; s_2, t_2)$ model and opens many possibilities of
further generalizations to other splitting mechanisms, based on different thermalizations (e.g. $SEP$ instead
of $SIP$ corresponding to ``maximal wealth'' restrictions).\par
In the  paper we restrict to models with two agents. However, all the results (symmetries and self-dualities) straightforwardly generalize to
a many-agent model, where each agent is associated to a vertex of an (undirected and simple) graph $G= (V,E)$   and,  independently and at exponential times, the two agents/nodes associated to each edge are updated according to the two-agent redistribution
rule. The product form of the (self-)duality functions allows these extensions.

This self-duality property is of great use if one wants to analyze the multi-agent model because the time dependent expectation of a multivariate polynomial of degree $k$ in the wealth of the different agents
will
be linked to the evolution of the total wealth of at most $k$ ``dual units'' - which is, of course, much simpler:
e.g., the expected wealth of one agent can simply be understood from the initial condition and a single continuous-time random walk.
Also, self-duality allows a quite complete characterization of the invariant measures of infinite systems (e.g.\ the continuous $IEM$ model), using properties - so-called
existence of a successful coupling - of the finite system (e.g.\ the discrete $IEM$ model) only.
Finally, taking a scaling limit (where the wealth of agent $i$
scales as $\lfloor Nx_i\rfloor$, with $N\to\infty$), one recovers from self-duality the duality between the discrete
and the continuous immediate exchange models.

The rest of our paper is organized as follows. In Section 2 we start with giving a new perspective on the $IEM$, by viewing its dynamics
as a composition of splitting, exchange and addition.
In Section 2.1 we study reversible measures and show how to recover the reversible measure of the $IEM$ from the reversible measure of the splitting part of the dynamics,
which is well-known because of its connection to the $SIP$. In Sections 2.2-2.3 we study the $SU(1,1)$-based symmetries of the splitting part and how to recover from them, by exchange and lumping,
the full $SU(1,1)$ symmetry of the $IEM$. As a consequence, in Theorem 2.1, we obtain self-duality of the $IEM$ from these symmetries. In Section 3 we study a generalized version of the $IEM$, and
also two new versions where the splitting part is based on thermalization of the $SEP$ and independent both symmetric and asymmetric random walkers, respectively. These new processes are studied along the same general scheme, but for different underlying algebras. 

\section{A new perspective on the immediate exchange model}
We start by reconsidering the dual immediate exchange model, introduced in \cite{grs}, which
is also a natural discrete analogue of the original model.
Here we have two agents with initial wealths $n_1,n_2\in \N$, where $\N=\{0,1,2,\ldots\}$ denotes
the set of non-negative integers.
In what follows we will denote by $\V_m$ the vector space of functions $f: \N^m\to\R$.

Our model is then described as follows: at the event times of a mean-one Poisson
process (or alternatively, at discrete times), the wealth is updated according to the following split, exchange and addition mechanism.
\ben [label={(\Roman*)}]
\item {\bf Splitting.} In the first step, the wealth of both agents is split according to
$n_1\mapsto (k_1, n_1-k_1)$ and
$n_2\mapsto (k_2, n_2-k_2)$, where $k_1$ (resp.\ $k_2$) are independent discrete uniform on
$\{0,\ldots, n_1\}$ (resp.\ $\{0,\ldots, n_2\}$). After this splitting we call
$k_1$ (resp.\ $k_2)$ the ``top part'' of the wealth of agent 1 (resp.\ agent 2); the remaining ones, i.e.\ $n_1-k_1$ and $n_2-k_2$, are called the ``bottom'' parts.\\
Let us denote by $X^{1,1}_{n_1,n_2}$ the $\N^4$-valued random variable with
distribution $(k_1, n_1-k_1; k_2, n_2-k_2)$ just described. Note that here the upper index $1, 1$ refers
to the choice of discrete uniforms for both $k_1$ and $k_2$.
This will be generalized later, where the distribution of $k_1$ can be chosen to be
Beta Binomial with parameters $n_1, s_1, t_1$ ($s_1=t_1=1$ corresponds to
the present uniform choice).\\
The splitting part of the dynamics can then simply be seen as the update from $(n_1,n_2)$ to
the four component random variable $X^{1,1}_{n_1,n_2}$.
\item {\bf Exchange.} In the second step, the top parts of both agents are exchanged,
i.e. $(k_1, n_1-k_1; k_2, n_2-k_2)$ goes to $(k_2, n_1-k_1; k_1, n_2-k_2)$.
This corresponds to the action of the so-called \emph{exchange map}
\[
\caE: \N^4\to\N^4: (n_{1,1}, n_{1,2}; n_{2,1}, n_{2,2})\mapsto (n_{2,1}, n_{1,2}; n_{1,1}, n_{2,2}),
\]
to which is associated a corresponding operator on functions $f\in \V_4$,
\[
\caE (f):= f\circ \caE.
\]
\item {\bf Addition.} At last, both parts of the wealth of each agent are added again, i.e.\ the final new wealths of both agents are
$$(k_2+n_1-k_1, k_1+n_2-k_2).$$
This corresponds to the surjective map
\[
\phi: \N^4\to \N^2: (n_{1,1}, n_{1,2}; n_{2,1}, n_{2,2})\mapsto (n_{1,1}+ n_{1,2}; n_{2,1}+ n_{2,2})
\]
and its corresponding operator $T_\phi: \V_2\to \V_4$, mapping functions from two variables to functions of four variables
via
\be\label{eee}
T_\phi f:= f\circ \phi
\ee
or, more explicitly, for $f:\N^2\to\R$, $T_\phi f: \N^4\to \R$ is defined via
\[
T_\phi f(n_{1,1}, n_{1,2}; n_{2,1}, n_{2,2})= f(n_{1,1}+ n_{1,2}, n_{2,1}+ n_{2,2})
\]
We note that if a function $g\in \V_4$ of four variables is in the image of $T_\phi$, i.e.\ it
is of the form $T_\phi f$ with
$f\in \V_2$, then on that function we can of course define $T_\phi^{-1}$ via $T_\phi^{-1} g:= f$ with $T_\varphi^{-1} T_\varphi= \1_{\V_2}$, the identity on $\V_2$.
The extension of $T_\phi^{-1}$ to the whole $\V_4$ is not unique, and we will later on make a particular choice in definition \ref{tphiminusone} below, which has
a natural connection with the redistribution of mass operator.
\een

With the notation introduced so far, we can describe one update in the $IEM(1,1)$ model as replacing
the initial wealth distribution  of the two agents (concentrated on $(n_1, n_2) \in \N^2$) by $\phi(\caE (X^{1,1}_{n_1,n_2}))$.
We can then write the generator of the immediate exchange model (abbreviation $IEM(1,1)$) as follows
\be\label{bombi}
L  = \Pi - \1_{\V_2},
\ee
where $\Pi$ is the transition operator on $\V_2$ described by
\be\label{piop}
\Pi f(n_1,n_2):= \E f(\phi(\caE (X^{1,1}_{n_1,n_2}))).
\ee

Furthermore, we introduce the so-called
  \emph{redistribution of mass operator} acting  on functions $f\in \V_4$,
\be\label{peop}
P f(n_{1,1},n_{1,2};n_{2,1},n_{2,2})=\sum_{k_1=0}^{n_1}\sum_{k_2=0}^{n_2}\frac1{n_1+1}\frac1{n_2+1} f(k_1,n_{1}-k_1;k_2,n_{2}-k_2),
\ee
with $n_1= n_{1,1}+n_{1,2}$, $n_2= n_{2,1}+n_{2,2}$.
Notice that $P: \V_4 \to \V_4$ maps a function $T_\varphi g$ with $g \in \V_2$ onto a function
of the form $T_\phi h$, for some $h\in \V_2$. Indeed, it is clear that the rhs of \eqref{peop} only depends
on $\phi(n_{1,1}, n_{1,2}; n_{2,1}, n_{2,2})= (n_1, n_2)$, provided $f(k_1, n_1-k_1; k_2, n_2-k_2)= g(n_1, n_2)$ for some $g \in \V_2$.
Moreover, via the  operators $P$, $\caE$ and $T_\varphi$,    an equivalent form for the transition operator $\Pi$ in \eqref{piop} is deduced:
\be\label{piop1}
\Pi f= T_\phi^{-1}(P (f\circ \phi\circ \caE))=T_\phi^{-1}P\caE T_\phi f.
\ee

In what follows, we will see that we can view $P$ as a thermalization of two $SIP(1)$ processes and,
as a consequence, the operator $\Pi$ will have symmetries (=commuting operators) arising from the ``addition'' (or ``lumping'', cf. section \ref{section symmetries})
of the symmetries of these two $SIP(1)$ generators, which will correspond to the
symmetries of a $SIP(2)$ process.

\subsection{Reversible measures for the process with transition operator $\Pi$}

By following the strategy presented e.g.\ in \cite{cggr2}, symmetries for $\Pi$ generate new self-duality functions  by acting on a cheap self-duality function related to the reversible measure of the process $\Pi$. Therefore, to prove self-duality,  we must look at the same time for both symmetries and reversible measures for $\Pi$.

Regarding the latter problem, one way out is to  solve directly a detailed balance equation (cf.\ \cite{grs}). As a result, the reversible product measures of the $IEM(1,1)$ model  are given by
products of discrete $\Gamma(2)$ (negative binomial with parameters $2,\lambda$) distributions, with marginals
\be\label{bobob}
\nu_\lambda (n)= (1-\lambda)^2\lambda^n (n+1),\ n \in \N,
\ee
where $0<\lambda<1$. If $X$ and $Y$ are i.i.d.\ with distribution \eqref{bobob}, then conditional
on the sum $X+Y=N$, the random variable $X$ is distributed according
to a Beta Binomial distribution with parameters $N, 2,2$.
Therefore, starting the $IEM(1,1)$ from an initial state $(n_1,n_2)$, it will converge to a
Beta Binomial distribution with parameters $N=n_1+n_2, 2,2$. More generally, we call the \emph{discrete $\Gamma(\beta,\lambda)$
distribution} the distribution for which
\be\label{boboba}
\nu^\beta_\lambda (n)= (1-\lambda)^{\beta}\frac{\lambda^n}{n!}\frac{\Gamma(n+\beta)}{\Gamma(\beta)},\ n \in \N,
\ee
with $0 < \lambda < 1$ and $\beta > 0$.
If $X$ and $Y$ are i.i.d.\ with distribution discrete $\Gamma(\beta,\lambda)$, resp.\ $\Gamma(\beta', \lambda)$, then
conditional
on the sum $X+Y=N$, the random variable $X$ is distributed according
to a Beta Binomial distribution with parameters $N, \beta,\beta'$.\\

Another possibility, which exploits the form \eqref{piop1} of $\Pi$, is to obtain reversible measures for $\Pi$ from those of $P$, the redistribution of mass operator, if available.
In order to do so, we
need a particular version of the map $T_\phi^{-1}$.

\begin{definition}\label{tphiminusone}
For any measure $\mu$ on $\N^4$, the operator $T_\varphi^{-1}: \V_4 \to \V_2$ in \eqref{piop1} is said to be \emph{$\mu$-canonical} if for any $f \in \V_4$
\begin{align*}
T_\varphi^{-1} f(n_1, n_2):= \underset{n_{2,1}+n_{2,2}=n_2}{\underset{n_{1,1}+n_{1,2}=n_1}{\sum_{(n_{1,1}, n_{1,2}; n_{2,1}, n_{2,2}) \in \N^4}}} f(n_{1,1}, n_{1,2}; n_{2,1}, n_{2,2}) \cdot \frac{\mu(n_{1,1}, n_{1,2}; n_{2,1}, n_{2,2}) }{\tilde \mu(n_1, n_2)},
\end{align*}
where $\tilde \mu:= \mu \circ \varphi$ is the image measure of $\mu$ under $\varphi$.
\end{definition}

Remark that saying that $T^{-1}_\varphi$ is $\mu$-canonical means
\begin{align*}
T^{-1}_\varphi f (n_1, n_2)= \E_\mu \left[ f(\cdot)\ \big|\ n_{1,1}+n_{1,2}=n_1, n_{2,1}+n_{2,2}=n_2  \right].
\end{align*}
We first prove some elementary properties
of this version of $T_\phi^{-1}$.
\begin{lemma}
Let $T_\varphi: \V_2 \to \V_4$ be the operator in \eqref{piop1} and $T^{-1}_\varphi$ be $\mu$-canonical for some $\mu$. Then we have:
\begin{itemize}
\item[(a)] $T^{-1}_\varphi T_\varphi f (n_1, n_2) = f (n_1, n_2)$, for  $f \in \V_2$ and $(n_1, n_2) \in \N^2$.
\item[(b)] $\int T_\varphi f d\mu = \int f d \tilde \mu$ for $f \in \V_2$.
\item[(c)] $\int T^{-1}_\varphi f d \tilde \mu = \int f d \mu$, for $f \in \V_4$.
\item[(d)] $T^{-1}_\varphi \left(T_\varphi f \cdot g \right)= f \cdot T^{-1}_\varphi g$, for $f \in \V_2$ and $g \in \V_4$.
\end{itemize}
\end{lemma}
\begin{proof}
Parts (a) and (b) are trivial. For part (c), simply by definition of $\varphi:= \N^4 \to \N^2$ and the law of total probability,
\begin{align*}
&\int T^{-1}_\varphi f d \tilde \mu=\\
&\quad=\sum_{(n_1, n_2) \in \N^2} \E_\mu\left[f\ \big|\ n_{1,1}+n_{1,2}=n_1, n_{2,1}+n_{2,2}=n_2\big| \right] \tilde \mu((n_1, n_2))\\
&\quad=
 \sum_{(n_1, n_2) \in \N^2} \E_\mu\left[f\ \big|\ n_{1,1}+n_{1,2}=n_1, n_{2,1}+n_{2,2}=n_2\big| \right]  \mu(\varphi^{-1}\{(n_1, n_2)\})\\
&\quad= \E_\mu\left[ f \right],
\end{align*}
where we remind that
\[
\varphi^{-1}(n_1, n_2):= \{(n'_{1,1}, n'_{1,2}, n'_{2,1}, n'_{2,2}): n'_{1,1}+n'_{1,2}=n_1, n'_{2,1}+n'_{2,2}=n_2 \} \subset \N^4.
\]
For part (d), for any $f \in \V_2$ and $g \in \V_4$ we have
\begin{align*}
&\left(T^{-1}_\varphi \left(T_\varphi f \cdot g \right) \right)(n_1, n_2)=\\
&\quad= \E_\mu\left[  (T_\varphi f) \cdot g\ \big|\ n_{1,1}+n_{1,2}=n_1, n_{2,1}+n_{2,2}=n_2 \right]\\
&\quad= f(n_1, n_2) \cdot\E_\mu\left[ g\ \big|\ n_{1,1}+n_{1,2}=n_1, n_{2,1}+n_{2,2}=n_2 \right]\\
&\quad= f(n_1, n_2) \cdot T^{-1}_\varphi g (n_1, n_2).
\end{align*}
\end{proof}

We discuss below the explicit condition to recover reversibility of the process $\Pi$ in terms of the reversible measure for $P$.

\begin{proposition}\label{proposition reversible}
Let $\mu$ be an invariant measure on $\N^4$ under the exchange map $\caE$, reversible for the process $P$, and assume moreover that
\begin{equation}\label{heart}
\Pi:= T^{-1}_\varphi P \caE T_\varphi = T^{-1}_\varphi \caE P T_\varphi,
\end{equation}
with $T^{-1}_\varphi$ being $\mu$-canonical.  Then $\tilde \mu:= \mu \circ \varphi$ is reversible for $\Pi$, i.e.
\begin{equation}\label{adj}
\Pi^\ast = \Pi,
\end{equation}
where $\Pi^\ast$ is the adjoint operator of $P$ in $L^2(\tilde{\mu})$.
\end{proposition}
\begin{proof}
First note that for all $f, g \in \V_4$,
\begin{equation}\label{invariantE}
\int f (\caE g) d \mu = \int (\caE f) g d \mu,
\end{equation}
by invariance of $\mu$ under $\caE$ and since $\caE^{-1}= \caE$.
Next, by reversibility of $\mu$, $P^\ast= P$, where $P^\ast$ is the adjoint in $L^2(\mu)$.

Therefore, for any $f, g \in \V_2$ we proceed as follows
\begin{align*}
&\int (\Pi f) g d \tilde \mu
= \int (T^{-1}_\varphi P \caE T_\varphi f) g d \tilde \mu\\
&\quad\overset{(a)}{=} \int (T^{-1}_\varphi P \caE T_\varphi f) (T^{-1}_\varphi T_\varphi g) d \tilde \mu
\overset{(d)}{=} \int T^{-1}_\varphi [(P \caE T_\varphi f)  (T_\varphi g)] d \tilde \mu\\
&\quad\overset{(c)}{=} \int (P\caE T_\varphi f) ( T_\varphi g) d \mu
\overset{\eqref{invariantE}}{=} \int (T_\varphi f) ( \caE P^\ast T_\varphi g) d \mu\\
&\quad = \int (T_\varphi f) ( \caE P T_\varphi g) d\mu
\overset{(c)}{=} \int T^{-1}_\varphi [(T_\varphi f ) (\caE P T_\varphi g)] d \tilde \mu\\
&\quad \overset{(c)}{=} \int (T^{-1}_\varphi T_\varphi f)(  T^{-1}_\varphi \caE P T_\varphi g) d \tilde \mu
\overset{(a)}{=} \int  f  (T^{-1}_\varphi \caE P T_\varphi g) d \tilde \mu\\
&\quad \overset{\eqref{heart}}{=} \int f  (T^{-1}_\varphi P \caE T_\varphi g) d \tilde \mu 
= \int f (\Pi g) d \tilde \mu,
\end{align*}
which concludes the proof.
\end{proof}

We conclude this section by providing a useful criterion for condition \eqref{heart} to hold. This criterion is the key to obtain reversible measures for any generalized immediate exchange model presented in this paper.

 We remark that
\begin{equation} \label{identity}
T_\varphi^{-1} T_\varphi f = f, \quad f \in \V_2,
\end{equation}
while in general $T_\varphi T_\varphi^{-1} g \neq g$ for some $g \in \V_4$.
\begin{proposition}\label{proposition 2}
If the redistribution operator $P$ is such that
\begin{equation} \label{aaa}
P = T_\varphi T^{-1}_\varphi,
\end{equation}
then condition \eqref{heart} holds.
\end{proposition}
\begin{proof}
The proof is straightforward by using  \eqref{aaa} and \eqref{identity},
\begin{align*}
T^{-1}_\varphi P \caE T_\varphi \overset{\eqref{aaa}}{=} T^{-1}_\varphi T_\varphi T^{-1}_\varphi \caE T_\varphi \overset{\eqref{identity}}{=} T^{-1}_\varphi \caE T_\varphi \overset{\eqref{identity}}{=} T^{-1}\caE T_\varphi T^{-1}_\varphi T_\varphi \overset{\eqref{aaa}}{=} T^{-1}_\varphi \caE P T_\varphi.
\end{align*}
\end{proof}

\subsection{Symmetries of the splitting part and connection with $SIP(1)$}
In order to find relevant symmetries of $\Pi$ in \eqref{piop}, it is now useful to understand
the connection between the redistribution of mass operator $P$ and the symmetric inclusion process, via
thermalization. See \cite{cggr} and \cite{gkrv} for more details on the notion of ``thermalization'' in the context
of models of heat conduction. The idea is to view the splitting of the wealth of an agent as running a $SIP(1)$ process for infinite time
(thermalization of $SIP(1)$) as we will now explain.

The $SIP(1)$ process on two sites is the process that makes jumps from state $(n,m)$ towards
$(n-1,m+1)$ at rate $n(1+m)$ and towards $(n+1,m-1)$ at rate $m(1+n)$; i.e.\ the process on $\N^2$ with generator
\beq\label{gensip2}
L^{SIP(1)} f(n,m)\nonumber
&=& n(1+m) (f(n-1,m+1)-f(n,m))
\nonumber\\
&+& m(1+n) (f(n+1,m-1)-f(n,m))
\eeq

The $SIP(1)$ process, when started from an initial state $(n,m)$, converges in the course of time
to $(k, n+m-k)$, where $k$ is  uniformly distributed on $\{0,\ldots, n+m\}$.
This implies that if we consider the first agent and initially put $n_1$ in its bottom part and
$0$ in its top part, running the $SIP(1)$ from that initial state for infinite time exactly produces  the splitting part for the first agent. By performing this operation for the two agents independently, we can rewrite
\begin{eqnarray}\label{bambi}
&&P f(n_{1,1},n_{1,2};n_{2,1},n_{2,2})\nonumber\\
&=&
\lim_{t\to\infty} \left(\E^{SIP(1)}_{n_{1,1},n_{1,2}}\otimes\E^{SIP(1)}_{n_{2,1},n_{2,2}}\right)
\left(f(n_{1,1}(t),n_{1,2}(t);n_{2,1}(t),n_{2,2}(t))\right).
\end{eqnarray}

We introduce the $K$-operators
\beq\label{bumbi}
K^+ f(n) &=& (1+n) f(n+1)\nonumber\\
K^- f(n) &=& n f(n-1)\nonumber\\
K^0 f(n) &=& \left(\tfrac12 + n \right) f(n)\ ,
\eeq
which form a (left) representation of the $SU(1,1)$ algebra, i.e.\
satisfy the commutation relations
\beq\label{commun}
&&[K^+, K^-]= 2K^0\nonumber\\
&&[K^{\pm}, K^0] = \pm K^{\pm}\ .
\eeq
In this representation, the generator of $SIP(1)$ is given by
\be\label{absip}
L^{SIP(1)}= K_1^+K_2^-+ K_1^-K_2^+ -2K^0_1K^0_2 + \frac12
\ee
where $K^\alpha_i$ denotes $K^\alpha$ working on the $i$-th variable $i\in \{1,2\},\alpha\in \{+,-,0\}$.
The form \eqref{absip}
is called the ``abstract'' form of the $SIP$ generator, and one easily
infers from it
the well known commutation property (see e.g. \cite{gkrv}), namely that it commutes
with $K_1^\alpha + K_2^\alpha$, $\alpha\in \{0,+,-\}$.

As a consequence of \eqref{bambi}, also the redistribution of mass operator $P$ commutes with
\[
\caK^\alpha:=K^\alpha_{1,1}+ K^\alpha_{1,2}+K^\alpha_{2,1}+K^\alpha_{2,2}.
\]

Because this ``symmetry'' of $P$ is the sum of four copies of the same operator, it is clear that
$\caK^\alpha$ is permutation invariant, and hence will also commute with the exchange operator $\caE$.
This is formalized in the following easy lemma.
\bl\label{perlem}
The symmetries $\caK^\alpha$ commute with the exchange operator $\caE$ and, as a consequence,
also with the operator $P\caE$.
\el
\bpr
Without loss of generality, we put ourselves in a two variable context and show the following.

Let $f:\N^2\to\R$ and denote $Ef(n,m):= f(m,n)$. Let $A:\V_1\to\V_1$ be an operator acting on functions of
one integer variable. Then the operator $\mathbb{A}:\V_2\to\V_2$ on functions of two variables, defined via
$\mathbb{A}:=A_1+A_2$, commutes with $E$. Here $A_1$ (resp.\ $A_2$) denotes the action of $A$ on the first
(resp.\ second) variable.
To show this, it suffices to prove that for functions of the form $f(n_1,n_2)= f_1(n_1)f_2(n_2)$ we have
$(\aaaa E-E\aaaa)f=0$. For such functions, $Ef(n_1,n_2)= f_1(n_2) f_2(n_1)$ and
\[
\aaaa f (n_1,n_2)= f_2(n_2)(Af_1)(n_1)+ f_1(n_1)(A f_2)(n_2),
\]
hence
\[
(\aaaa Ef) (n_1,n_2)= f_1(n_2) (A f_2)(n_1)+ f_2(n_1) (Af_1)(n_2) = (E \aaaa f)(n_1, n_2).
\]
\epr

\subsection{Additive structure of symmetries and self-duality}\label{section symmetries}

The representation \eqref{bumbi} of $SU(1,1)$ has a particular parameter which was set equal to one  but that can be assumed to be a general positive constant, giving the
one-parameter family of discrete representations defined by
\beq\label{genbumbi}
K^{+,\kappa} f(n) &=& (\kappa + n) f(n+1)\nonumber\\
K^{-,\kappa} f(n) &=& n f(n-1)\nonumber\\
K^{0,\kappa} f(n) &=& \left(\tfrac{\kappa}{2}+n\right) f(n)\ .
\eeq
These operators satisfy the same $SU(1,1)$ commutation relations for all $\kappa>0$ and
the choice we need to make for the immediate exchange model $IEM(1,1)$ is $\kappa=1$.

The above $K$-operators have a natural additive structure which is expressed in the following lemma, whose proof
is an easy computation left to the reader.
\bl\label{fifi}
Define the \emph{lumping operator} $\caA:\V_1\to\V_2$ from functions of one variable to functions of two variables
via
\[
(\caA f)(n,m)= f(n+m).
\]
We denote its \emph{generalized inverse} $\caA^{-1}: \text{Im}(\caA)\to \V_1$ to be the operator which maps a two-variable function $\tilde f\in \V_2$ in the image
of $\caA$, i.e.\ a function of the form $\tilde f(n,m)=f (n+m)$ with $f \in \V_1$, to the
one-variable function $f \in \V_1$.
Then we have for all $\kappa_1,\kappa_2 >0$, $\alpha\in \{0,+,-\}$, $f:\N\to\R$ and $n_1,n_2\in\N$,
\beq\label{bulu}
((K^{\alpha,\kappa_1}_1+ K^{\alpha, \kappa_2}_2) \caA f)(n_1,n_2)&=& (K^{\alpha,\kappa_1+\kappa_2} f)(n_1+n_2)
\nonumber\\
&=&
(\caA (K^{\alpha,\kappa_1+\kappa_2} f))(n_1,n_2)\ ,
\eeq
which can be rewritten as
\be\label{commukap}
\caA^{-1}(K_1^{\alpha,\kappa_1}+ K_2^{\alpha, \kappa_2}) \caA =  K^{\alpha,\kappa_1+\kappa_2}.
\ee
\el

Now consider an operator $B$ which acts on functions of two variables $\tilde f\in\V_2$ and such that
for functions of the form $\tilde f(n,m)= f (n+m)$, i.e.\ $\tilde f=\caA f$,  it is of the form
$B \tilde f (n,m)= (\tilde B f) (n+m)$, i.e.\ it conserves functions in the image of the lumping operator.
We call such an operator ``lumpable'' and its associated one-variable operator $\tilde{B}:\V_1\to\V_1$ the
``lumped operator of $B$'' (cf. Figure 1).
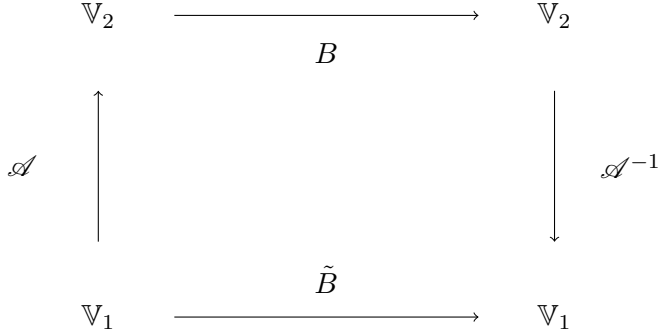
\begin{figure}
\begin{center}
\begin{tikzpicture}
\node at (0,0) {$\V_2$};
\node at (6,0) {$\V_2$};
\node at (0,-4) {$\V_1$};
\node at (6,-4) {$\V_1$};
\node at (3, -.5) {$B$};
\node at (3, -3.5) {$\tilde B$};
\node at (-1, -2) {$\caA$};
\node at (7, -2) {$\caA^{-1}$};
\draw [->] (1,0) -- (5,0);
\draw [->] (6,-1) -- (6,-3);
\draw [->] (0,-3) -- (0,-1);
\draw [->] (1,-4) -- (5,-4);
\end{tikzpicture}
\end{center}
\caption{Lumping of $B$ with respect to $\caA$.}
\end{figure}

Then, as a consequence of Lemma
\ref{fifi},  we have that
\be\label{Bop}
(K^{\alpha,\kappa_1}_1+ K^{\alpha, \kappa_2}_2) B \tilde f (n_1,n_2)=
(K^{\alpha,\kappa_1+\kappa_2} \tilde B  f)(n_1+n_2).
\ee
As a consequence, if $B$ commutes with $(K^{\alpha,\kappa_1}_1+ K^{\alpha, \kappa_2}_2) $, then
the ``lumped operator'' $\tilde{B}$ commutes with the ``lumped symmetry''  $K^{\alpha,\kappa_1+\kappa_2}$.

This can be understood from the fact that if $B,C:\V_2\to \V_2$ are lumpable operators with
corresponding $\tilde{B}, \tilde{C}:\V_1\to\V_1$, then we can write
\[
\tilde{B}= \caA^{-1} B\caA, \quad \tilde{C} = \caA^{-1} C\caA,
\]
and observe that whenever an operator $W$  is lumpable w.r.t.\ $\caA$, we have
$\caA\caA^{-1} W \caA= W\caA$,
from which it follows that $$[\tilde B, \tilde C]= \caA^{-1}[B,C]\caA,$$ where $[B,C]=BC-CB$ denotes the commutator.\\

This can now be applied  to the process with transition operator $\Pi$, because  $P\caE$ is
a lumpable operator (where we go now from four to two variables via the map $T_\phi$ defined in \eqref{eee}, which
is the analogue of the lumping operator of Lemma \ref{fifi}).
Indeed, recall that $P\caE f(n_{1,1},n_{1,2}; n_{2,1}, n_{2,2})=\E f(\caE(X^{1,1}_{n_1,n_2})$ and,
since the distribution of $X^{1,1}_{n_1,n_2}$ depends only on $n_1, n_2$, so
does the distribution of $\caE(X^{1,1}_{n_1,n_2})$. Therefore, after taking expectations,
we are still left with  $P \caE f$ being a function of $n_1$ and $n_2$ only.

So we obtain the following
theorem. Notice that this theorem was already proved in \cite{grs} with the help of
direct somewhat tedious computations with hypergeometric functions.
\bt\label{11thm}
The generator of the immediate exchange model $IEM(1,1)$ commutes with
\[
K^{\alpha, 2}_1+ K^{\alpha, 2}_2
\]
for $\alpha\in \{+,-,0\}$.
As a consequence, the immediate exchange model is self-dual with self duality functions
\[
D(k_1,k_2;n_1,n_2)= d(k_1,n_1) d(k_2,n_2),
\]
where
\be\label{hihi}
d(k,n)=
\begin{cases}
1  &\text{if}\ k=0\\
0 &\text{if}\ k>n \\
\frac{n!}{(n-k)!}\frac{\Gamma\left(2\right)}{\Gamma\left(2+k\right)}
&\text{otherwise}\ .
\end{cases}
\ee
\et
\bpr
The commutation property comes from the commutation of $P\caE$ with
$\caK^\alpha:=K^{\alpha, 1}_{1,1}+ K^{\alpha,1}_{1,2}+ K^{\alpha, 2}_{2,1}+ K^{\alpha, 2}_{2,2}$
(cf. Lemma \ref{perlem}),  the lumpability of $P \caE$ and
the addition of the representation indices $\kappa$ (in this concrete setting $1+1=2$) for the $K$-operators
(cf. Lemma \ref{fifi}).

The fact that
these symmetries lead to the self-duality function \eqref{hihi} follows
from the general strategy for obtaining self-duality functions from symmetries
in \cite{cggr2} and \cite{gkrv}. In particular, the self-duality function \eqref{hihi} arises by
acting with the symmetry $e^{K^{+,2}_1+K^{+,2}_2}$ on the cheap self-duality function
coming from the reversible measure (which we showed in \cite{grs} to be the product of two discrete $\Gamma(2)$ distributions).
\epr
\section{Further examples and generalizations}
Before we introduce the general $IEM(s_1,t_1;s_2,t_2)$ model, we introduce a slight generalization of the
$SIP(1)$ process, namely the $SIP(s,t)$ process for  $s, t > 0$.
\bd
The $SIP(s,t)$ process is the process on $\N^2$ with generator
\beq\label{sipst}
L^{SIP(s,t)} f(n,m) \nonumber
&=& n(t+m) (f(n-1,m -1)-f(n,m))
\nonumber\\
&+& m(s+n) (f(n+1,m-1)- f(n,m)).
\eeq

\ed
In terms of the abstract generator representation \eqref{absip}, we have now that

\be\label{absipst}
L^{SIP(s,t)}= K_1^{+,s}K_2^{-,t}+ K_1^{-,s}K_2^{+,t} -2K^{0,s}_1K^{0,t}_2 + \frac{st}{2} \1_1 \1_2,
\ee
where $K^{\alpha,s}$ are defined in \eqref{genbumbi}.
As a consequence, this generator commutes with
\[
K_1^{\alpha,s}+K_2^{\alpha,t}.
\]

\subsection{Generalized immediate exchange  model $IEM(s_1,t_1;s_2,t_2)$}
Unlike the discrete ``homogeneous'' model where the wealth $n_1$ (resp.\ $n_2$) of the first (resp.\ second) agent is uniformly on ${0, \ldots, n_1}$ (resp.\ on $\{0, \ldots, n_2\}$) redistributed over
the two ``pockets'', now the redistribution is
Beta Binomial with parameters $n_1, s_1,t_1$ for the first agent (and
independent Beta Binomial with parameters $n_2, s_2,t_2$ for the second one).

We denote as before the four dimensional random variable $X^{s_1,t_1;s_2,t_2}_{n_1,n_2}$,
which corresponds to the splitting of $(n_1,n_2)$.
This splitting results from the thermalization of $SIP(s_1,t_1)$ for the first agent
and $SIP(s_2,t_2)$ for the second agent.

Moreover, if we keep the same definitions of the mappings $\caE$ and $\varphi$, then the $IEM(s_1,t_1;s_2,t_2)$ model has a one-step transition operator given by
\be\label{pigen}
\Pi f(n_1,n_2)= \E f\circ \phi\circ\caE (X^{s_1,t_1;s_2,t_2}_{n_1,n_2})= \left(T_{\phi}^{-1}P\caE T_\phi f\right) (n_1,n_2),
\ee
where, as before, $\E$ denotes expectation and $P$ is the analogue of the redistribution of mass operator
\eqref{peop} in this context.
In the notation of \eqref{genbumbi}, the operators
\be\label{cak}
\caK^{\alpha}= K^{\alpha, s_1}_{1,1}+ K^{\alpha,t_1}_{1,2} + K_{2,1}^{\alpha, s_2} + K_{2,2}^{\alpha, t_2}
\ee
are symmetries of $P$, for $\alpha \in \{+, -, 0\}$.
However, these symmetries commute with $\caE$ if and only if $s_1=s_2$, i.e.\
the parameters of the representations of $SU(1,1)$ for the sites where the exchange takes place have to be the same.
As a consequence we have the following
analogue of theorem \ref{11thm}.
\bt\label{stthm}
If $s_1=s_2=s$, the
transition operator $\Pi$ of \eqref{pigen}
commutes with the operators
\be\label{kasop}
K^{\alpha, s_1+t_1}_1+ K^{\alpha, s_2+t_2}_2.
\ee
As a consequence, $IEM(s,t_1;s,t_2)$ is self-dual with self-duality functions
\[
D(k_1,k_2; n_1,n_2)= d_{s+t_1}(k_1, n_1) d_{s+t_2} (k_2,n_2)
\]
where
\be\label{hihir}
d_r(k,n)=
\begin{cases}
1  &\text{if}\ k=0\\
0  &\text{if}\ k>n \\
\frac{n!}{(n-k)!}\frac{\Gamma\left(r\right)}{\Gamma\left(r+k\right)} &\text{otherwise}\ .
\end{cases}
\ee
\et
\bpr
Since the $K$-operators in \eqref{cak}
are symmetries of $P$, if $s_1=s_2$ they are also symmetries of
$\caE$. By lumpability of these symmetries (cf.  \eqref{Bop} and text around it) we obtain
the commutation  of the transition operator $\Pi$ with the operators \eqref{kasop}.

Concerning the reversible measure of $\Pi$,  we simply observe that the process related to the  redistribution of mass operator $P$, obtained as a thermalization of two $SIP(s, t_1)$ and $SIP(s, t_2)$ processes, admits as reversible measure  the product measure $\mu:= \nu_\lambda^s \otimes \nu_\lambda^{t_1} \otimes \nu_\lambda^s \otimes \nu_\lambda^{t_2}$ of four discrete Gamma distributions $\nu_\lambda^\beta$, for any $0 < \lambda < 1$ and suitable parameters $\beta$. We also observe that $\mu$ is permutation invariant since $s_1=s_2=s$. Moreover, we know that the image measure $\tilde \mu:= \mu \circ \varphi$ is simply the product measure of two discrete $\Gamma(s+t_1, \lambda)$ and $\Gamma(s+t_2, \lambda)$ distributions. At last, since the redistribution of mass via $P$ is made according to this ergodic measure $\mu$ and the particle conservation of  the $SIP$ process,  we have that $P= T_\varphi T_\varphi^{-1}$ given $T_\varphi^{-1}$ is $\mu$-canonical. From Propositions \ref{proposition 2} and \ref{proposition reversible}, we then conclude that $\tilde \mu$ is reversible for $\Pi$. Therefore,  we can use the cheap self-duality function associated
to this reversible measure and acting on it
with the operator $\exp({K_1^{+, s+t_1}+ K_2^{+, s+t_2}})$
as in the proof of theorem 2.1, we obtain the self-duality \eqref{hihir}.
\epr
\subsection{Models based on SEP thermalization}
We consider a different splitting mechanism in these models. Here the initial wealth of both
agents is first redistributed over two ``pockets'' having both
a maximal capacity, i.e.\ they contain a fixed number of ``slots'' in which only one ``coin'' at the time can fit. Given this ``pocket-structure'', each agent places one  coin at the time in one of the non-occupied slots, uniformly chosen among the two pockets.   Therefore, the model
has four positive integers as parameters, also referred to as capacities. We call it the ``restricted immediate exchange model'' abbreviated ``$RIEM$".

More precisely, for the first agent,  say,  $n_1$ coins are redistributed over
two pockets with maximal capacities $\gamma_1$ and $\delta_1$ according
to a hypergeometric distribution with parameters $n_1,\gamma_1,\delta_1$, i.e.\ $n_1\leq \gamma_1+\delta_1$ is split in $(k_1,n_1-k_1)$ where
$k_1$ has distribution
\be\label{hypergeo}
\pee (k_1=m)= \frac{{\gamma_1\choose m}{\delta_1\choose n_1-m}}{{\gamma_1+\delta_1\choose n_1}} 1_{m\leq \gamma_1}.
\ee

In the $RIEM(\gamma_1,\delta_1; \gamma_2,\delta_2)$
model the random variable $X^{\gamma_1,\delta_1; \gamma_2,\delta_2}_{n_1,n_2}$
describing the splitting of the initial wealth $(n_1,n_2)$ of the two agents
is $(k_1,n_1-k_1; k_2, n_2-k_2)$ where $k_1,k_2$ are independent
hypergeometric distribution with parameters $n_1,\gamma_1,\delta_1$ and $n_2,\gamma_2,\delta_2$, respectively.
We also denote as before the maps
$\caE,\phi$.

Then the $RIEM(\gamma_1,\delta_1;\gamma_2,\delta_2)$ model has a one-step transition operator given by
\be\label{pigensep}
\Pi f(n_1,n_2)= \E f\circ \phi\circ\caE (X^{\gamma_1,\delta_1;\gamma_2,\delta_2}_{n_1,n_2})= \left(T_{\phi}^{-1}P\caE T_\phi f\right) (n_1,n_2),
\ee
where as before $\E$ denotes expectation and $P$ denotes the analogue of mass redistribution operator
\eqref{peop} in this context.
In analogous notation of \eqref{genbumbi}, the operators
\be\label{caj}
\caJ^{\alpha}= J^{\alpha, \gamma_1}_{1,1}+ J^{\alpha,\delta_1}_{1,2} + J_{2,1}^{\alpha, \gamma_2} +
J_{2,2}^{\alpha, \delta_2}
\ee
are symmetries of $P$ for $\alpha \in \{+,-,0\}$.
Here the $J^{\alpha,\gamma}$ are the operators working on functions
$f: \{0,\ldots,\gamma\}\to\R$ defined via
\beq\label{su2gen}
J^{+,\gamma} f(n) &=& (\gamma-n) f(n-1)\nonumber\\
J^{-,\gamma} f(n) &=& n f(n-1)\nonumber\\
J^{0,\gamma} f(n) &=& \left(\tfrac{\gamma}{2}-n\right) f(n)\ ,
\eeq
generators of a (left) representation of the $SU(2)$ algebra.

These symmetries commute with the exchange operator $\caE$ if and only if $\gamma_1=\gamma_2$, i.e.\
the parameters of the representations of $SU(2)$ for the sites where the exchange takes place have to be the same.
Moreover, these $J$-operators have the same additive structure of the $K$-operators considered above, i.e.\
$(J^{\alpha, \gamma_1}_{1,1}+ J^{\alpha,\delta_1}_{1,2} )f(n_{1,1}+n_{1,2})= (J_1^{+,\gamma_1+\delta_1} f) (n_1)$,
where $n_1=n_{1,1}+n_{1,2}$.

As a consequence, we obtain the following
analogue of Theorem \ref{stthm}. The proof, which is
a copy of the proof of the $SU(1,1)$ case, replacing $K$-operators
by $J$-operators, is left to the reader.
\bt\label{gdthm}
If $\gamma_1=\gamma_2=\gamma$, the
transition operator \eqref{pigensep}
commutes with the operators
\[
J^{\alpha, \gamma_1+\delta_1}_1+ J^{\alpha, \gamma_2+\delta_2}_2.
\]
As a consequence, $RIEM(\gamma,\delta_1;\gamma,\delta_2)$ is self-dual with self-duality functions
\[
D(k_1,k_2; n_1,n_2)= d_{\gamma+\delta_1}(k_1, n_1) d_{\gamma+\delta_2} (k_2,n_2)
\]
where
\be\label{hihirr}
d_r(k,n)=
\begin{cases}
1  &\text{if}\ k=0\\
0 &\text{if}\ k>n\ \text{or}\ n>r \\
\frac{{n\choose k}}{{r\choose k}} &\text{otherwise}\ .
\end{cases}
\ee
where ${a\choose b}$ is defined to be zero if $a<b$.
\et
\subsection{Model based on independent random walks}
In this model, the splitting process is the thermalization of
independent symmetric walkers, with Poisson distributions as invariant measures.
To describe the model, for each $n \in \N$
we denote  by $X_n$ a random variable
with distribution
\be\label{poiss}
\pee( X_n=k) = \frac{\frac{1}{k!}\frac{1}{(n-k)!}}{Z_n}= {n\choose k} \frac{1}{2^n},
\ee
where $Z_n$ is a normalizing constant.
This distribution is the first marginal of two independent Poisson random variables (with same parameter) conditioned on their sum to be equal
to $n$, which is the binomial $n$ with success probability $1/2$. Because independent random walkers have Poisson distributions as their invariant measures, this distribution can be obtained
from thermalizing independent random walkers.

More precisely, defining the process on $\N^2$ with generator
\beq\label{indrwgen}
L^{RW} f(n,m) \nonumber &=& n(f(n-1, m+1)-f(n,m))
\nonumber\\
&+& m (f(n+1,m-1)-f(n,m)),
\eeq
we have that
\be\label{indtherm}
\lim_{t\to\infty} \E_{n,m}f(n(t), m(t))= f(X_{n+m}, n+m-X_{n+m})
\ee
where $X_{\cdot}$ is the random variable defined via \eqref{poiss}.

We can then define the independent random walk redistribution model as follows:
\ben [label={(\roman*)}]
\item Start from two agents with initial wealth $n_1,n_2\in \N$.
\item Split the wealth into four components (two pockets for each agent)
$n_{1,1}= X_{n_1}, n_{1,2}= n_1- X_{n_1}$, $n_{2,1}= X_{n_2}, n_{2,2}= n_2-X_{n_2}$ where
$X_{n_1}, X_{n_2}$ are independent random variables of  distribution \eqref{poiss}.
\item Exchange the wealth $n_{1,1}$ and $n_{2,1}$, i.e.\
obtain the exchanged four tuple $(X_{n_2}, n_1- X_{n_1}, X_{n_1}, n_2-X_{n_2})$.
\item Add the wealth of the two pockets again to obtain the
final new wealths of the two agents: $(X_{n_2}+ n_1- X_{n_1}, X_{n_1}+ n_2-X_{n_2})$.
\een
This procedure gives then
the one-step transition operator
\be\label{irwtrans}
\Pi f(n_1,n_2)= \E f(X_{n_2}+ n_1- X_{n_1}, X_{n_1}+ n_2-X_{n_2}),
\ee
where the expectation $\E$ is w.r.t.\ the two independent random variables $X_{n_1}, X_{n_2}$.
We then have the following self-duality result.
\bt\label{indrwthm}
The process with transition operator $\Pi$ given by \eqref{irwtrans} is self-dual
with self-duality function
\be\label{bombio}
D(k_1,k_2;n_1,n_2)= \frac{n_1! n_2!}{(n_1-k_1)! (n_2-k_2)!}.
\ee
\et
\bpr
As before, we first look at the symmetries of the four-variable transition operator
\[
P: \V_4\to \V_4
\]
defined via
\be\label{bababo}
P f(n_{1,1}, n_{1,2}, n_{2,1}, n_{2,2}):= \E f( X_{n_1}, n_1-X_{n_1}, X_{n_2}, n_2- X_{n_2}),
\ee
where $\E$ is w.r.t.\ the two independent random variables $X_{n_1}, X_{n_2}$ and $n_1= n_{1,1}+ n_{1,2}, n_2= n_{2,1}+ n_{2,2}$.
To prove the theorem, we show that
\ben [label={(\arabic*)}]
\item $P$ commutes with  the two symmetries
\[
S^\dagger= a^\dagger_{1,1}+a^\dagger_{1,2}+ a^\dagger_{2,1}+ a^\dagger_{2,2},\quad S^-=  a_{1,1}+a_{1,2}+ a_{2,1}+ a_{2,2},
\]
where
\be\label{adag}
a^\dagger f(n)= f(n+1), \quad a f(n)= n f(n-1).
\ee
This follows from the thermalization procedure \eqref{indtherm} and the fact that
the generator $L^{RW}$ of \eqref{indrwgen} commutes with   $a^\dagger_1+a^\dagger_2$ and $a_1+a_2$ (this  follows from its representation
\[
L^{RW}= -(a_1-a_2) (a_1^\dagger-a_2^\dagger)
\]
and the commutation relations $[a_i^\dagger,a_j]=\delta_{i,j} \1$, where $\1$ is the identity).
\item The symmetries $S^\dagger$ and $S^-$ commute with the exchange operator. This follows as in Lemma \ref{perlem} from the fact that the symmetries are
sums of copies of the same operator working on different variables.
\item On functions of the form $ \tilde f(n_{1,1}, n_{1,2}, n_{2,1}, n_{2,2})= f (n_1,n_2)$, the symmetries
act as
\beq
S^\dagger \tilde f(n_{1,1}, n_{1,2}, n_{2,1}, n_{2,2}) &=& (\caS^\dagger f )(n_1,n_2)\nonumber\\
S^- \tilde f(n_{1,1}, n_{1,2}, n_{2,1}, n_{2,2}) &=& (\caS^- f )(n_1,n_2)\ ,
\eeq
where
\[
\caS^\dagger:=  a_1^\dagger+ a_2^\dagger,\quad  S^-:= a_1+ a_2.
\]
This is an immediate consequence of the form \eqref{adag} of the operators $a^\dagger, a$.
\een
The result then follows because the duality function \eqref{bombio} is obtained from acting with the symmetry $e^{a_1^\dagger+a_2^\dagger}$ on the cheap self-duality function
\[
D_{cheap} (k_1,k_2;n_1,n_2)= k_1! k_2! \delta_{k_1,n_1}\delta_{k_2,n_2},
\]
which is recovered, as in the proof of Theorem \ref{stthm}, by checking conditions in Propositions \ref{proposition reversible} and \ref{proposition 2}
for the reversible product measure of Poisson distributions for $P$.
\epr

\subsection{Model based on asymmetric random walks}
A generalization of the previous model is obtained when we consider (possibly) asymmetric walkers for the thermalization.
On a two vertex system, consider the generator
\beq\label{lalpha}
L^{RW(q)}&=& q (a_2^\dagger a_1- a_2 a_2^\dagger) + (a_1^\dagger a_2 - a_1 a_1^\dagger) \nonumber\\
&=& - (a_1 - a_2) (a_1^\dagger - q a_2^\dagger)
\eeq
working on functions $f\in \V_2$ and with $q > 0$. This represents independent random walkers jumping with rate $q$ for jumps from site $1$ to site $2$ and
rate one for jumps from $2$ to $1$.
The reversible measure is given by the product
\be\label{prodpois}
\rho_\lambda\otimes\rho_{q \lambda},
\ee
where $\rho_\lambda$ is the Poisson distribution
with parameter $\lambda > 0$.
A commuting operator is given by $S^\dagger=a_1^\dagger+ a_2^\dagger$, i.e.\ $[L^{RW(q)},S^\dagger]=0$.

Let us call $(X,Y)$ a pair of independent random variables jointly distributed as \eqref{prodpois}; then, conditional on $X+Y=n$,
$X \sim \text{Bin}(n, 1/(1+q))$  and $Y=n-X$.
Therefore, following the setup of the previous models,  for $n_1,n_2\in\N$, $q_1, q_2>0$, we define the four-dimensional random variable
$X^{q_1,q_2}_{n_1,n_2}$ responsible for the redistribution of mass as $(n_{1,1}, n_{1,2}; n_{2,1},n_{2,2})$, where $n_{1,1}$ and $n_{2,1}$ are independent
$\text{Bin}(n_1, 1/(1+q_1))$ and  $\text{Bin}(n_2, 1/(1+q_2))$, and $n_{1,2}= n_1-n_{1,1}, n_{2,2}= n_2-n_{2,1}$.

Denoting as before $\caE$ the exchange map, we can then define the transition operator $\Pi$ of the ``Poissonian'' immediate exchange model
$PIEM(q_1, q_2)$ via
\be\label{pipiem}
\Pi f(n_1,n_2)= \E \left(f(n_{2,1}+n_{1,2}, n_{1,1}+ n_{2,2})\right).
\ee

To uncover the relevant symmetries of $\Pi$, we note that $\Pi= T_\varphi^{-1}P\caE T_\varphi$, where $P$ is as before the redistribution operator
\[
Pf (n_{1,1}, n_{1,2}; n_{2,1}, n_{2,2})= \E f(X^{q_1, q_2}_{n_{1,1}+n_{1,2}, n_{2,1}+n_{2,2}}).
\]
A symmetry of $P$ is given by $$\caS^\dagger:= a_{1,1}^\dagger+a_{1,2}^\dagger+a_{2,1}^\dagger+a_{2,2}^\dagger.$$
This symmetry also commutes with $\caE$ and, therefore,
it commutes with $P\caE$. Moreover, $\caS^\dagger$ is lumpable w.r.t.\ $T_\varphi$, and the lumped symmetry is $S^\dagger:= 2 (a_1^\dagger+a_2^\dagger )$, as it can be seen from the simple identity:
\[
(a_{1,1}^\dagger+a_{1,2}^\dagger+a_{2,1}^\dagger+a_{2,2}^\dagger) f(n_{1,1}+n_{1,2}, n_{2,1}+n_{2,2})= 2 (a_1^\dagger + a_2^\dagger) f (n_1, n_2),
\]
with $n_1= n_{1,1}+n_{1,2}, n_2=n_{2,1}+n_{2,2}$.
Therefore, the lumped operator $\Pi= T_\varphi^{-1} P\caE T_\varphi$ commutes with $S^\dagger=2(a_1^\dagger+ a_2^\dagger)$.
To apply this symmetry in order to produce a useful self-duality function we have to start, as usual,  from the cheap self-duality function corresponding
to a reversible measure for the process with transition operator $\Pi$.
This is provided by the following lemma.
\bl
For all $\lambda>0$, the Poisson product measure
\be\label{rev}
\rho_{\lambda(1+q_1)}\otimes \rho_{\lambda(1+q_2)}
\ee
is reversible for the transition operator $\Pi$.
\el
\bpr
First we remark that the product measure $\mu:= \rho_\lambda\otimes \rho_{\lambda q_1}\otimes \rho_{\lambda'}\otimes\rho_{\lambda' q_2}$
is reversible for the redistribution operator $P$. Therefore, if additionally $\lambda=\lambda'$, it is reversible for
the operator $P\caE$. Moreover, by choosing the $\mu$-canonical $T_\varphi^{-1}$, condition \eqref{aaa} holds and hence, by Proposition \ref{proposition reversible}, the measure \eqref{rev} is obtained as image measure $\mu \circ \varphi$.
\epr

The self-duality result is immediately derived by acting with $e^{\frac12 S^\dagger}$ on the cheap self-duality function associated to the reversible measure for $\Pi$ in \eqref{rev}, after setting the parameter $\lambda$ equal to $1$.
\begin{theorem}
The process with transition operator $\Pi$ in \eqref{pipiem}, associated to two systems of $q_1$ and $q_2$-asymmetric random walks as above, is self-dual with self-duality function
\be
D_{q_1, q_2}(k_1, k_2; n_1, n_2):= d_{q_1}(k_1, n_1) d_{q_2}(k_2, n_2),
\ee
where
\begin{align*}
d_{q}(k,n):= \begin{cases}
0 &\text{if } k > n\\
\frac{n!}{(n-k)!}\frac{1}{e^{-(1+q)}  (1+q)^{n}} &\text{if }  k \leq n\ .
\end{cases}
\end{align*}
\end{theorem}

%
%
%
%
%

\end{document}